\documentclass[12pt,a4paper]{article}
\usepackage{hyperref}

\usepackage{diagbox}
\usepackage{pifont}

\usepackage[T1]{fontenc}
\usepackage[latin1]{inputenc}
\usepackage{geometry}
\usepackage{float}
\usepackage{mathrsfs}
\usepackage{amsmath}
\usepackage{amssymb}
\usepackage{graphicx}

\usepackage{subcaption}

\hypersetup{
    colorlinks=true,
    linkcolor=blue,
    filecolor=magenta,      
    urlcolor=cyan,
    citecolor = red,
}

\makeatletter

\usepackage{algorithm,algpseudocode}

\usepackage{amsthm}

\usepackage{mathrsfs}

\usepackage{amsfonts}

\usepackage{epsfig}

\usepackage{bm}

\usepackage{mathrsfs}

\usepackage{enumerate}

\@ifundefined{definecolor}{\@ifundefined{definecolor}
 {\@ifundefined{definecolor}
 {\usepackage{color}}{}
}{}
}{}
\newtheorem{ass}{Assumption}[section]

\newcommand{\bbR}{\mathbb R}

\newtheorem{theorem}{Theorem}[section]

\newtheorem{prop}{Proposition}[section]

\newtheorem{ex}{Example}[section]
\newcounter{hypA}

\newcounter{hypB}

\newcounter{hypD}

\usepackage{babel}

\makeatother

\newcommand{\bbE}{\mathbb{E}}
\newcommand{\bbP}{\mathbb{P}}

\newcommand{\bbZ}{\mathbb{Z}}
\newcommand{\bbN}{\mathbb{N}}

\newcommand{\cA}{\mathcal{A}}

\newcommand{\cN}{\mathcal{N}}

\newcommand{\cM}{\mathcal{M}}

\newcommand{\cG}{\mathcal{G}}
\newcommand{\sX}{\mathsf X}

\newcommand{\cI}{{\cal I}}

\newcommand{\ba}{\bm{\alpha}}
\newcommand{\bx}{\bm{x}}
\newcommand{\bZ}{\bm{Z}}

\newcommand{\cC}{\mathcal{C}}
\newcommand{\cO}{\mathcal{O}}

\newcommand\norm[1]{\left\lVert#1\right\rVert}

\begin{document}

%+Title
\begin{center}

{\Large \textbf{A randomized multi-index sequential Monte Carlo method}}

\vspace{0.5cm}

BY 
XINZHU LIANG, %$^{1}$
SHANGDA YANG, %$^{1}$, 
SIMON L. COTTER, %$^{1}$
KODY J.~H. LAW %$^{1}$,

{\footnotesize %$^{1}$
School of Mathematics, University of Manchester, Manchester, M13 9PL, UK.}
{\footnotesize E-Mail:\,} \texttt{\emph{\footnotesize 
xinzhu.liang@postgrad.manchester.ac.uk,
shangda.yang@manchester.ac.uk,
simon.cotter@manchester.ac.uk,
kody.law@manchester.ac.uk}}\\
%neil.walton@mancheter.ac.uk, 
%{\footnotesize $^{2}$Computer, Electrical and Mathematical Sciences and Engineering Division, King Abdullah University of Science and Technology, Thuwal, 23955, KSA.}
%{\footnotesize E-Mail:\,} \texttt{\emph{\footnotesize ajay.jasra@kaust.edu.sa}}
\end{center}

\begin{abstract}
    We consider the problem of estimating expectations with respect to a target distribution with an unknown normalizing constant, and where even the unnormalized target needs to be approximated at finite resolution. Under such an assumption, this work builds upon a recently introduced multi-index Sequential Monte Carlo (SMC) ratio estimator, which provably enjoys the complexity improvements of multi-index Monte Carlo (MIMC) and the efficiency of SMC for inference. The present work leverages a randomization strategy to remove bias entirely, which simplifies estimation substantially, particularly in the MIMC context, where the choice of index set is otherwise important. Under reasonable assumptions, the proposed method provably achieves the same canonical complexity of MSE$^{-1}$ as the original method (where MSE is mean squared error), but without discretization bias. It is illustrated on examples of Bayesian inverse and spatial statistics problems. 
\end{abstract}

%\keywords{Bayesian Inverse Problems,  Sequential Monte Carlo, Multi-Index Monte Carlo}



%\maketitle

\section{Introduction}

We consider the computational approach to Bayesian inverse problems \cite{stuart2010inverse},
which has attracted a lot of attention in recent years. One typically requires the expectation of a quantity of interest $\varphi(x)$,  
where unknown parameter $x \in \sX \subset \bbR^{d}$ 
has posterior probability distribution $\pi(x)$, as given by Bayes' theorem
\footnote{We allow the case 
$d \rightarrow \infty$}. 
Assuming the target distribution cannot be computed analytically, we instead compute the expectation as
\begin{equation}
    \int_{\sX}\varphi(x)\pi(dx) = \frac{1}{Z} \int_{\sX}\varphi(x)f(x) \pi_0(dx),
\end{equation}
where $Z = \int_{\sX}f(x)\pi_0(dx)$, $f$ is prescribed 
up to a normalizing constant, $\pi_0$ is the prior and $\pi(x) \propto f(x) \pi_0(x)$. 
Markov chain Monte Carlo (MCMC) \cite{geyer1992practical,robert1999monte,bernardo1998regression,cotter2013mcmc} and sequential Monte Carlo (SMC) \cite{chopin2020introduction,del2006sequential} are two methodologies which can be used to achieve this. 
In this paper, we use the degenerate notations $d\pi(x) = \pi(dx) = \pi(x)dx$ to mean the same thing, i.e. the probability under $\pi$ of an infinitesimal volume element $dx$ (Lebesgue measure by default) centered at $x$.

Standard Monte Carlo methods can be %highly 
costly, particularly in the case where the problem involves 
approximation of an underlying continuum domain,
a problem setting which is becoming progressively more prevalent over time \cite{tarantola2005inverse,cotter2013mcmc,stuart2010inverse,law2015data,van2015nonlinear,oliver2008inverse}. The multilevel Monte Carlo (MLMC) method was developed to reduce the computational cost in this setting, by performing most simulations with low accuracy and low cost \cite{giles2015multilevel}, and successively refining the approximation with corrections that use fewer simulations with higher cost but lower variance. The MLMC approach has attracted a lot of interest from those working on inference problems recently, such as
MLMCMC \cite{dodwell2015hierarchical,hoang2013complexity} and MLSMC \cite{beskos2018multilevel,beskos2017multilevel,moral2017multilevel}. The related {\em multi-fidelity} Monte Carlo methods often focus on the case where the models lack structure 
and quantifiable convergence behaviour 
\cite{peherstorfer2018survey,cai2022multi},
which is very common across science and engineering applications. It is worthwhile to note that MLMC methods can be implemented on the same class of problems, and {\em do not require structure or a priori convergence 
estimates in order to be implemented}. However, convergence rates provide a convenient mechanism to deliver quantifiable theoretical results. This is the case for both multilevel and multi-fidelity approaches.

Recently, an extension of the MLMC method has been established called multi-index Monte Carlo (MIMC) \cite{haji2016multi}. Instead of using first-order differences, MIMC uses high-order mixed differences to reduce the variance of the hierarchical differences dramatically. MIMC has first been considered to apply to the inference context in \cite{ourmismc2} and \cite{jasra2018multi,jasra2021multi}. The state-of-the-art research of MIMC in inference is presented in \cite{law2022multi}, in which the MISMC ratio estimator for posterior inference is given, and the theoretical convergence rate of this estimator is guaranteed.
Although a canonical complexity of MSE$^{-1}$ for the MISMC ratio estimator can be achieved, it is still suffers from discretization bias. This bias constrains the choice of index set and estimators of the bias often suffer from high variance, which means implementation can be cumbersome for challenging problems where the method is expected to be particularly advantageous otherwise. 

Debiasing techniques were first introduced in \cite{rhee2012new,rhee2015unbiased}, \cite{mcleish2011general} and \cite{strathmann2015unbiased}, with many more works using or developing it further \cite{agapiou2014unbiased,glynn2014exact,jacob2015nonnegative,lyne2015russian,walter2017point}. These debiasing techniques are based on a similar idea as MLMC, but in addition to reducing the estimator variance, the former focus on building unbiased estimators. The connection between the debiasing technique and the MLMC method has been pointed out by \cite{dereich2015general}, \cite{giles2015multilevel} and \cite{rhee2015unbiased}. \cite{vihola2018unbiased} has further clarified the connection within a general framework for unbiased estimators. The first work to combine the debiasing technique and MLMC in the context of inference is \cite{chada2021unbiased}. A recent breakthrough involves using double randomization strategies to remove the bias of the increment estimator \cite{heng2021unbiased,jasra2021unbiased,jasra2020unbiased}. 

The starting point of our current work is the MISMC ratio estimator introduced in \cite{law2022multi}. Our new {\em randomized MISMC (rMISMC) ratio estimator} will be reformulated in the framework of \cite{rhee2015unbiased} to remove discretization bias entirely. Like the MISMC ratio estimator, our estimator provably enjoys the complexity improvements of MIMC and the efficiency of SMC for inference. Theoretical results will be given to show that it achieves the canonical complexity of MSE$^{-1}$ under appropriate assumptions, but {\em without any discretization bias and the consequent requirements for its estimation}. 
From a practical perspective, estimating this bias, and balancing it along with the variance and cost in order to select the index set, comprises a significant overhead for existing multi-index methods.
In addition to convenience and simplification, 
the particular formulation of our un-normalized estimators is novel, and may prove useful in other contexts where one cannot obtain i.i.d. samples from the increments. 
The unbiased estimators of the normalizing constant and un-normalized integral 
can also be useful in their own right, 
in the context of 
Robbins-Monro \cite{robbins1951stochastic}
or other stochastic approximation algorithms \cite{kushner2012stochastic, law2019estimation, jasra2021unbiased}. 

The paper is organized as follows. In Section~\ref{sec:problems}, we present the motivating problems considered in the following numerical experiments. In Section~\ref{sec:rMISMC}, the original MISMC ratio estimator is reviewed for convenience, and the rMISMC ratio estimator and its theoretical results are stated. In Section~\ref{sec:numerics}, we apply MISMC and rMISMC methods on Bayesian inverse problems for elliptic PDEs and log Gaussian process models.

\section{Motivating problems}
\label{sec:problems}

Here, we introduce the Bayesian inference for a D-dimensional elliptic partial differential equation and two statistical models, the log Gaussian Cox model and the log Gaussian process model. We will apply the methods that we present in Section~\ref{sec:rMISMC} to these motivating problems in order to show their efficacy.

\subsection{Elliptic partial differential equation}
We consider a D-dimensional elliptic partial differential equation defined over an open domain $\Omega \subset \bbR^{D}$ with locally continuous boundary $\partial \Omega$, i.e. the boundary is the graph of a continuous function in a neighbourhood of any point. 
Given a forcing function ${\sf f}(x): \Omega \rightarrow \bbR$
and a diffusion coefficient function
$a(x): \Omega \rightarrow \bbR$, 
depending on a random variable $x \sim \pi$, 
the partial differential equation for $u(x): \bar{\Omega} \rightarrow \bbR$ (where $\bar{\Omega}$ is the closure of $\Omega$) is given by
\begin{equation}
\begin{aligned}
    -\nabla \cdot (a(x) \nabla u(x)) & = {\sf f}(x), \ \ \text{on} \ \Omega,\\
    u(x)& = 0, \ \ \text{on} \ \partial \Omega \, .
\label{eqn:pde}
\end{aligned}
\end{equation}
The dependence of the solution $u$ of~(\ref{eqn:pde}) on $x$ is raised from the dependence of $a$ and ${\sf f}$ on $x$.

In particular, we assume the prior distribution in the numerical experiment as
\begin{equation}
    x \sim U[-1,1]^d =: \pi_0
\end{equation}
and $a(x)$ as 
\begin{equation}\label{eqn:coef_pde}
    a(x)(z) = a_0 + \sum_{i=1}^{d} x_i \psi_i(z),
\end{equation}
where $\psi_i$ are smooth functions with $\norm{\psi_i}_{\infty}:= 
 \sup_{z\in \Omega} \vert \psi(z) \vert \leq 1$  for $i=1,...,d$, and $a_0 > \sum_{i=1}^{d}x_i$.

\subsubsection{Finite element method}
\label{sec:pde_finite}

Consider 1D piecewise linear nodal basis functions $\phi_{j}^{K}$ for meshes $\{ z_{i}^{K} = i/(K+1) \}_{i=0}^{K+1}$, $j=1,2,...,K$, which is defined as
\begin{equation}\label{pde_basis}
\phi_j^{K}(z) = \left \{
\begin{array}{lll}
    \frac{z-z_{j-1}^{K}}{z_{j}^{K}-z_{j-1}^{K}}, &  z \in [z_{j-1}^{K}, z_{j}^{K}]\\
    1-\frac{z-z_{j-1}^{K}}{z_{j+1}^{K}-z_{j}^{K}}, &  z \in [z_{j}^{K}, z_{j+1}^{K}] \\
    0, & \text{otherwise}.
\end{array}
\right.
\end{equation}

For an index $\alpha = (\alpha_1,\alpha_2) \in \bbN^{2}$, we can form the tensor product grid over $\Omega =[0,1]^2$ as 
\begin{equation}
    \{(z_{i_1}^{K_{1,\alpha}},z_{i_2}^{K_{2,\alpha}})\}^{K_{1,\alpha}+1,K_{2,\alpha}+1}_{i_1=0,i_2=0},
\end{equation}
where $K_{1,\alpha}=2^{\alpha_1}$ and $K_{2,\alpha}=2^{\alpha_2}$ and the mesh size in each direction is $K_{1,\alpha}^{-1}$ and $K_{2,\alpha}^{-1}$, respectively. Then the bilinear basis function is constructed by the product of nodal basis functions in two directions:
\begin{equation}
\phi_i^{\alpha}(z) = \phi_{i_1,i_2}^{\alpha}(z_1,z_2) = \phi_{i_1}^{K_{1,\alpha}}(z_1)\phi_{i_2}^{K_{2,\alpha}}(z_2),
\end{equation}
where $i=i_1+K_{1,\alpha}i_2$ for $i_1=1,...,K_{1,\alpha}$ and $i_2 = 1,...,K_{2,\alpha}$ and $K_{\alpha} = K_{1,\alpha}K_{2,\alpha}$.

A Galerkin approximation can be written as 
\begin{equation}
    u_{\alpha}(x)(z) = \sum_{i=1}^{K_{\alpha}}u_{\alpha}^{i}(x)(z) \
    \phi_{i}^{\alpha}(z),
\end{equation}
where $u_{\alpha}^i$ for $i=1,...,K_{\alpha}$ are approximate values of the solution $u(x)$ at mesh points that we want to obtain. Using Galerkin approximation to solve the weak solution of PDE~(\ref{eqn:pde}), we can derive a corresponding Galerkin system:
\begin{equation}
    {\bf A}_{\alpha}(x) {\bf u}_{\alpha}(x) = {\bf f}_{\alpha}(x),
\end{equation}
where ${\bf A}_{\alpha}(x)$ is the stiffness matrix whose components are given by
\begin{align}
&  [{\bf A}_{\alpha}(x)]_{ij} := \int_{z_{j_1-1}^{K_{1,\alpha}}}^{z_{j_1+1}^{K_{1,\alpha}}}
\int_{z_{j_2-1}^{K_{2,\alpha}}}^{z_{j_2+1}^{K_{2,\alpha}}} 
a(x)(z) \nabla \phi_{i}^{\alpha}(z) \cdot \nabla\phi_{j}^{\alpha}(z) dz,
\end{align}
where $j = j_1+j_2 K_{1,\alpha}$ for $j_1 = 1,...,K_{1,\alpha}$ and $j_2  =1,...,K_{2,\alpha}$, 
$${\bf u}_{\alpha}(x) = [u_{\alpha}^{1}(x)(z_{1}),...,u_{\alpha}^{K_{\alpha}}(x)(z_{K_{\alpha}})]^{T},$$ 
and $$[{\bf f}_{\alpha}(x)]_{j} = \int_{z_{j_1-1}^{K_{1,\alpha}}}^{z_{j_1+1}^{K_{1,\alpha}}}\int_{z_{j_2-1}^{K_{2,\alpha}}}^{z_{j_2+1}^{K_{2,\alpha}}} {\sf f}(x)(z) \phi_{j}^{\alpha}(z) dz.$$ 

\subsubsection{The Bayesian inverse problem}
 Under an elliptic partial differential equation model, we wish to infer the unknown parameter value $x \in \sX \subset \bbR^{d}$ given $n$ evaluations of the solution $y \in \bbR^{n}$ \cite{stuart2010inverse}. We aim to analyse the posterior distribution $\bbP(x \vert y)$ with density $\pi(x) = \pi(x \vert y)$. In practice, one can only expect to evaluate a discretized version of $x\in \sX$. $\pi(dx)$ then can be obtained up to a constant of proportionality by applying Bayes' theorem:
\begin{equation}
    \pi(dx) \propto L(x) \pi_0(dx),
\end{equation}
where $\pi_0(dx)$ is a density of the prior distribution and $L(x)$ is the likelihood which is proportional to the probability density of the data $y$ was created with a given value of the unknown parameter $x$. 

Define the vector-valued function as follows
\begin{equation}
    \cG(u(x)) = [v_1(u(x)),...,v_n(u(x))]^{T},
\end{equation}
where $n$ is the number of data, $v_i \in L^2$ and $v_i(u(x)) = \int v_i(z)u(x)(z)dz$ for $i=1,...,n$. Then the data can be modelled as
\begin{equation}
    y = \cG(u(x)) + \nu, \ \ \ \nu \sim \cN(0,\Sigma),
\end{equation}
where $\cN(0,\Sigma)$ denotes the Gaussian distribution with mean zero and variance-covariance matrix $\Sigma$. Then the likelihood of the evaluations $y$ can be derived as 
\begin{equation}
    \pi(y\vert x) \propto L(x) := \exp \Big(-\frac{1}{2}\vert y-\cG(u(x))\vert ^2_{\Sigma} \Big),
\end{equation}
where $\vert w \vert_{\Sigma} = (w^{\top}\Sigma^{-1}w)^{1/2}$.

When the solution of the elliptic PDE can only be solved approximately, we denote the approximate solution at resolution multi-index $\alpha$ as $u_{\alpha}$ as described above and the approximate likelihood is given by 
\begin{equation}\label{eqn:finite_likelihood_pde}
    \pi_{\alpha}(y\vert x) \propto L_{\alpha}(x)  := \exp \Big(-\frac{1}{2}\vert y-\cG(u_{\alpha}(x))\vert ^2_{\Sigma} \Big),
\end{equation}
and the posterior density is given by
\begin{equation}\label{eqn:likelihood_pde}
    \pi_{\alpha}(dx) \propto L_{\alpha}(x)\pi_0(dx).
\end{equation}

\subsection{Log Gaussian process models}
\label{sec:lgp}

Now, we consider the log Gaussian Cox model and the log Gaussian process model. A log Gaussian process (LGP) $\Lambda(z)$ is given by
\begin{equation}
    \Lambda(z) = \exp\{ x(z) \} \, ,
\end{equation}
where $x = \{ x(z): z\in \Omega \subset \bbR^{D} \}$ is a real-valued Gaussian process \cite{rasmussen2003gaussian,stuart2010inverse}.
The log Gaussian process model provides a flexible  approach to non-parametric density modelling with controllable smoothness properties. However, inference for the LGP is intractable. 
The LGP model for density estimation 
\cite{tokdar2007posterior} assumes data $z_i \sim p$, where $p(z) = \Lambda(z)/\int_{\Omega} \Lambda(z) dz$. As such, the likelihood of $x$ associated to observations $\mathsf{Z}=\{z_1,\dots,z_n\}$ is given by
\begin{equation}\label{eqn:likelihood_lgp}
L(x ; \mathsf{Z} ) =  
\prod_{z \in \mathsf{Z}} p(z) = 
\left(\int_{\Omega} \Lambda(z) dz \right)^{-n} 
\prod_{z \in \mathsf{Z}} \Lambda(z) \, .
\end{equation}

The log Gaussian Cox (LGC) model assumes the observations are distributed according to a spatially inhomogeneous Poisson point process with intensity function given
by $\Lambda$.
The likelihood of observing $\mathsf{Z}=\{z_1,\dots,z_n\}$
under the LGC model is \cite{moller1998log,murray2010elliptical,law2022multi,cai2022multi} 
\begin{equation}\label{eqn:likelihood_lgc} 
\begin{aligned}
  L(x ; \mathsf{Z}) & =  \exp\Big\{ \int_{\Omega} (1-\Lambda(z)) dz \Big\} \prod_{z\in \mathsf{Z}} \Lambda(z), \\
  &= \exp\Big\{ \int_{\Omega} (1-\exp(x(z))) dz \Big\} 
  \prod_{z\in \mathsf{Z}} \exp(x(z))\, .
\end{aligned}
\end{equation}
This construction has an elegant simplicity, which is flexible and convenient due to the underlying Gaussian process.
Some example applications are presented in
\cite{diggle2013spatial}.

We consider a dataset comprised of the location of $n=126$ Scots pine saplings in a natural forest in Finland \cite{moller1998log}, denoted $z_1,...,z_n \in [0,1]^2$. 
This is modeled with both LGC, following \cite{heng2020controlled},
and LGP, following \cite{tokdar2007posterior}. The prior is defined in terms of a KL-expansion with a suitable parameter $\theta = (\theta_1,\theta_2,\theta_3)$ as follows, for $z \in [0,2]^2$,
\begin{equation}\label{eqn:gaussian_pro}
    x(z) = \theta_1 + \sum_{k \in \bbZ \times \bbZ_{+} \cup \bbZ_{+} \times 0} \rho_{k} (\theta) (\xi_{k}\phi_{k}(z) + \xi_{k}^{*} \phi_{-k}(z)), 
\end{equation}
$$\xi_{k} \sim \cC\cN(0,1) \ \text{i.i.d.},$$
where $\cC\cN(0,1)$ denotes a standard complex normal distribution, $\xi_{k}^{*}$ is the complex conjugate of $\xi_{k}$, $\phi_{k}(z) \propto \exp[\pi i z \cdot k]$ are Fourier series basis functions (with $i=\sqrt{-1}$) and 
\begin{equation}\label{eqn:spectral_decay}
    \rho_{k}^{2}(\theta) = \theta_2/((\theta_3 + k_1^2)(\theta_3 + k_2^2))^{\frac{r+1}{2}}.
\end{equation}
The coefficient $r$ controls the smoothness, and here we will choose $r=1.6$. Note that the periodic prior measure is defined on $[0,2]^2$ so that no boundary conditions are imposed on the sub-domain $[0,1]^2$. Then, the posterior distribution is given by 
\begin{equation}
    \pi(dx) \propto L(x) \pi_0(dx),
\end{equation}
where $\pi_0$ is constructed in (\ref{eqn:gaussian_pro}) and $L(x)$ is constructed  in (\ref{eqn:likelihood_lgc}) (or (\ref{eqn:likelihood_lgp})).

\subsubsection{The finite approximation problem}

One typically use a grid-based approximation to approximate the inferences in LGC \cite{murray2010elliptical,diggle2013spatial,teng2017bayesian,cai2022multi} and in LGP \cite{riihimaki2014laplace,griebel2010finite,tokdar2007towards}. We approximate the likelihoods and priors of LGC and LGP by the fast Fourier transform (FFT) respectively, as described below. First, we truncate the KL-expansion of prior as follows, for an index $\alpha = (\alpha_1,\alpha_2) \in \bbN^{2}$,
\begin{equation}\label{eqn:gaussian_pro_app}
    x_{\alpha}(z) = \theta_1 + \sum_{k \in \cA_{\alpha}} \rho_{k} (\theta) (\xi_{k}\phi_{k}(z) + \xi_{k}^{*} \phi_{-k}(z)), 
\end{equation}
$$\xi_{k} \sim \cC\cN(0,1) \ \text{i.i.d.},$$
where $\cA_{\alpha} := \{ -2^{\alpha_1/2},-(2^{\alpha_1/2}-1),...,2^{\alpha_1/2}-1,2^{\alpha_1/2} \} \times \{ 1,2,...,2^{\alpha_2/2}-1,2^{\alpha_2/2} \} \cup \{ 1,2,...,2^{\alpha_2/2}-1,2^{\alpha_2/2} \} \times 0$. The cost for approximating $x_{\alpha}(z)$ over the grid is $\cO((\alpha_1+\alpha_2)2^{\alpha_1+\alpha_2})$. The finite approximations of the likelihood of LGC and LGP are then defined by
\begin{align}
    & \text{(LGC)} \ 
    L_{\alpha}(x_{\alpha}) :=\exp \Bigg[ \sum_{i=1}^{n} \hat{x}_{\alpha}(z_i) - Q(\exp(x_{\alpha})) \Bigg] ,\\
    & \text{(LGP)} \ 
    L_{\alpha}(x_{\alpha}) :=\exp \Bigg[ \sum_{i=1}^{n} \hat{x}_{\alpha}(z_i) - n\log Q(\exp(x_{\alpha})) \Bigg] \label{eqn:finite_likelihood_lgp},
\end{align}
where $\hat{x}_{\alpha}(z)$ is defined as an interpolant over the grid output from FFT and $Q$ denotes a quadrature rule,
such that $Q(\exp (x_{\alpha})) \approx \int_{[0,1]^2} \exp(x(z)) dz$. Then, the finite approximations of the posterior distribution of LGC and LGP are defined by
\begin{equation}
    \pi_{\alpha}(dx_{\alpha}) \propto L_{\alpha}(x_{\alpha}) \pi_0(dx_{\alpha}).
\end{equation}

The quantity of interest for these models will be $\varphi(x) = \int_{[0,1]^2} \exp (x(z)) dz$, and we will estimate
its expectation $\pi(\varphi) = \bbE ( \varphi(x) \vert z_1,\dots, z_n )$.

\section{Randomized Multi-index sequential Monte Carlo}
\label{sec:rMISMC}

The original MISMC estimator has been considered in \cite{ourmismc2} and \cite{jasra2018multi,jasra2021multi}. Convergence guarantees have been established in \cite{law2022multi}, which demonstrates the importance of selecting a reasonable index set, by comparing the results with  the tensor product index set and  the total degree index set. Then, a very interesting extension to multi-index sequential Monte Carlo is introduced, which is called randomized multi-index sequential Monte Carlo. The basic methodology of randomized multi-index Monte Carlo is first introduced in \cite{rhee2012new,rhee2015unbiased}. Instead of giving an index set in advance, we choose $\alpha$ randomly from a distribution. Another advantage of this approach is that it can give an unbiased unnormalized estimator, which is discretization-free.

Define the target distribution as 
$\pi(x) = f(x)/Z$, where $Z = \int_{\sX} f(x) dx$
and $f(x):= L(x) \pi_0(x)$. Given a quantity of interest $\varphi: \sX \rightarrow \bbR$, for simplicity, we define 
\begin{equation}
    f(\varphi) := \int_{\sX} \varphi(x) f(x) dx = f(1) \pi(\varphi),
\end{equation}
where $f(1) = \int_{\sX} f(x) dx = Z$. 
Define their approximations at finite resolution 
$\alpha \in \bbZ_+^D$
by
$\pi_\alpha(x) = f_\alpha(x)/Z_\alpha$, where $Z_\alpha = \int_{\sX} f_\alpha(x) dx$
and $f_\alpha(x):= L_\alpha(x) \pi_0(x)$, and $\varphi_{\alpha}:\sX \rightarrow \bbR$, where 
$\lim_{\vert\alpha\vert\uparrow \infty} f_\alpha = f$
and $\lim_{\vert\alpha\vert\uparrow \infty}
\varphi_\alpha = \varphi$.

Consider the ratio decomposition
\begin{equation}\label{eq:increments_ratio}
\pi(\varphi) = \frac{f(\varphi)}{f(1)} = 
\frac{\sum_{\alpha\in \bbZ_+^D} \Delta( f_\alpha(\varphi_\alpha) )} 
{\sum_{\alpha\in \bbZ_+^D} \Delta f_\alpha(1)},
\end{equation}
where $\Delta$ is the first-order mixed difference operator 
\begin{equation}\label{eq:increment}
\Delta = \otimes_{i=1}^{D} \Delta_i := \Delta_D \circ \cdots \circ \Delta_1 \, ,
\end{equation}
which is defined recursively by the first-order difference operator $\Delta_i$ along direction $1 \leq i \leq D$. If $\alpha_i>0$, 
\begin{equation}\label{eq:incrementi}
\Delta_i \varphi_{\alpha}(x_{\alpha}) = \varphi_{\alpha}(x_{\alpha}) - \varphi_{\alpha-e_i}(x_{\alpha}) \, ,
\end{equation}
where $e_i$ is the canonical vectors in $\bbR^{D}$, i.e. $(e_i)_j=1$ for $j=i$ and 0 otherwise. If $\alpha_i=0$, $\Delta_i \varphi_{\alpha}(x_{\alpha}) = \varphi_{\alpha}(x_{\alpha})$.

For convenience, we denote the vector of multi-indices
\begin{equation}
    \ba(\alpha) := (\ba_1(\alpha),...,\ba_{2^D}(\alpha) ) \in \bbZ_{+}^{D \times 2^D}, 
\end{equation}
where $\ba_1(\alpha) = \alpha$, $\ba_{2^D}(\alpha) = \alpha - \sum_{i=1}^{D}e_i$ and $\ba_i(\alpha)$ for $1 < i < 2^D$ correspond to the intermediate multi-indices while computing the mixed difference operator $\Delta$.

Throughout this section $C>0$ is a constant whose value may change from line to line.

\subsection{Original MISMC ratio estimator}

In order to make use of \eqref{eq:increments_ratio}, we need to construct estimators of 
$\Delta( f_\alpha(\zeta_\alpha) )$, both for our quantity of interest 
$\zeta_\alpha = \varphi_\alpha$ and for $\zeta_\alpha=1$. The natural and naive way to estimate $\Delta( f_\alpha(\zeta_\alpha) )$ is based on sampling from a coupling of $(\pi_{\ba_1(\alpha)},...,\pi_{\ba_{2^D}(\alpha)})$. However, this is not a trivial approach, instead we construct an approximate coupling $\Pi_{\alpha}: \sigma(\sX^{2^D}) \rightarrow [0,1]$ as follows. We first define the coupling prior distribution as 
\begin{equation}\label{eq:priorcoupling}
\Pi_0(d\bx) = \pi_0(d\bx_1) \prod_{i=2}^{2^D} \delta_{\bx_1}(d \bx_i) \, ,
\end{equation}
where $\bx = (\bx_1,...,\bx_{2^D}) \in \sX^{2^D}$ and $\delta_{\bx_1}$ denotes the Dirac delta function at $\bx_1$.
Note that this is an exact coupling of the prior in the sense that 
for any $j \in \{1,\dots, 2^D\}$
\begin{equation}\label{eq:priormarginal}
\int_{\sX^{2^D-1}} \Pi_0(d\bx_{-j}) 
= \pi_{0}(d\bx_j) \, .
\end{equation}
Here we denote $\bx_{-j} = (\bx_1,...,\bx_{j-1},\bx_{j+1},...,\bx_{2^D})$ which omits the $j$th coordinate. Indeed it is the same coupling used in MIMC \cite{haji2016multi}.

In order to provide estimates analogous to the variance rate in the MIMC \cite{haji2016multi}, we use the SMC sampler \cite{chopin2020introduction,del2006sequential} to compute. We hence adapt Algorithm \ref{algo:smc} to an extended target which is an approximate coupling 
of the actual target as in \cite{jasra2018bayesian,jasra2018multi,jasra2021multi}, \cite{ourmismc2} and \cite{vihola}, 
and utilize a ratio of estimates, similar to
\cite{vihola}.
To this end, we define a likelihood on the coupled space as
\begin{equation}\label{eq:coupledlike}
{\bf L}_\alpha(\bx) = \max \{ L_{\ba_1(\alpha)}(\bx_1), \dots, L_{\ba_{2^D}(\alpha)}(\bx_{2^D}) \}  \, .
\end{equation}
The approximate coupling is defined by
\begin{equation}\label{eq:approxcoupling}
F_\alpha(d\bx) = {\bf L}_\alpha(\bx) \Pi_0(d\bx), \qquad \Pi_\alpha(d\bx) = \frac1{F_\alpha(1)} F_\alpha(d\bx) \, .
\end{equation}

\begin{ex}[Approximate Coupling]\label{ex:approx_coupling}
Let $D=2$, $d=1$ and $\alpha=(1,1)$, an example of the approximate coupling constructed in \eqref{eq:priorcoupling}, \eqref{eq:coupledlike} and \eqref{eq:approxcoupling} is given by, for $\bx =(\bx_1,\bx_2,\bx_3,\bx_4) \in \sX^{4}$,
\begin{align*}
    \Pi_{(1,1)}(\bx_1,\bx_2,\bx_3,\bx_4)  & \approx F_{(1,1)}(\bx_1,\bx_2,\bx_3,\bx_4) = {\bf L}_{(1,1)}(\bx_1,\bx_2,\bx_3,\bx_4) \Pi_0(\bx_1,\bx_2,\bx_3,\bx_4),
\end{align*}
where ${\bf L}_{(1,1)}(\bx_1,\bx_2,\bx_3,\bx_4) = \max\{ L_{00}(\bx_1),L_{01}(\bx_2),L_{10}(\bx_3),L_{11}(\bx_4) \}$ and $ \Pi_0(\bx_1,\bx_2,\bx_3,\bx_4) = \pi_0(\bx_1)\delta_{\bx_1}(\bx_2) \delta_{\bx_1}(\bx_3) \delta_{\bx_1}(\bx_4)$. For our choice of prior coupling \eqref{eq:priorcoupling}, we effectively have a single distribution, for $x \in \sX$,
\begin{equation*}
    \Pi_{(1,1)} \approx \max \{ L_{00}(x), L_{01}(x),L_{10}(x),L_{(11)}(x) \} \pi_0(x).
\end{equation*}
Note that any suitable prior which preserves the marginal as in \eqref{eq:priormarginal} is admissible.
\end{ex}

Let $H_{\alpha,j} = F_{\alpha, j+1}/F_{\alpha,j}$ for some intermediate distributions
$F_{\alpha,1}, \dots, F_{\alpha,J}=F_{\alpha}$, for example $F_{\alpha,j} = {\bf L}_\alpha(\bx)^{\tau_j} \Pi_0(\bx)$,
where the tempering parameter satisfies $\tau_1=0$, $\tau_j<\tau_{j+1}$, and $\tau_J=1$ 
(for example $\tau_j = (j-1)\tau_0$).
Now let $\bm{\cM}_{\alpha,j}$ for $j=2,\dots,J$
be Markov transition kernels such that $(\Pi_{\alpha,j} \bm{\cM}_{\alpha,j})(d\bx) = \Pi_{\alpha,j}(d\bx)$, analogous to $\cM$ as any suitable MCMC kernel \cite{geyer1992practical,robert1999monte,cotter2013mcmc}.

\begin{algorithm}
\caption{SMC sampler for coupled estimation of $\Delta( f_\alpha(\zeta_\alpha) )$}\label{algo:smc}
\begin{algorithmic}[1]
\State Let $\bx^{1,i} \sim \pi_1$ for 
$i=1,\dots, N$, and $\bZ^N_1=1$. For $j=2,\dots,J$, repeat the following steps for $i=1,\dots, N$:
\State Store $\bZ_j^N = \bZ_{j-1}^N \frac1N \sum_{n=1}^N H_{\alpha,j-1}(\bx^{j-1,n})$.
\State Define $w_j^i = H_{\alpha,j-1}(\bx^{j-1,i}) / \sum_{k=1}^N H_{\alpha,j-1}(\bx^{j-1,k})$.
\State Resample. Select $I_j^i \sim \{ w_j^1, \dots, w_j^N\}$, and let $\hat{\bx}^{j,i} = \bx^{j-1,I_j^i}$.
\State Mutate. Draw $\bx^{j,i} \sim \bm{\cM}_{\alpha,j}( \hat{\bx}^{j,i},\cdot)$.
\end{algorithmic}
\end{algorithm}

For $j=1,\dots, J$ define 
\begin{equation}\label{eq:nocoup}
\Pi^N_{\alpha,j}(d\bx) := \frac1N \sum_{i=1}^N \delta_{\bx^{j,i}}(d\bx) \, , 
\end{equation}
and then define
\begin{equation}\label{eq:unnocoup}
\bZ_{\alpha}^N := \prod_{j=1}^{J-1} \Pi^N_{\alpha,j}( H_{\alpha,j} ) \, , \quad  
F^N_\alpha(d\bx) := \bZ_\alpha^N \Pi^N_{\alpha,J}(d\bx) \, .
\end{equation}

The following Assumption will be needed.
\begin{ass} \label{ass:G1}
Let $J\in\mathbb{N}$ be given, and let $\sX$ be a Banach space. For each $j\in\{1,\dots,J\}$ there exists some $C>0$ such that for all $(\alpha,x)\in\mathbb{Z}_+^D\times \sX$, 
$$
C^{-1} < Z, L_\alpha(x) 
\leq C.
$$
\end{ass}

Then, we have the following convergence result
\cite{del2004feynman}. 
\begin{prop}\label{prop:unnocoupled}
Assume \ref{ass:G1}. Then for any $(J,p)\in\mathbb{N}\times(0,\infty)$ there exists a $C>0$ such that for any $N\in\mathbb{N}$, $\psi : \sX^{2^D} \rightarrow \bbR$ bounded and
measurable, and $\alpha\in\mathbb{Z}_+^D$,
$$
\bbE\left[\vert F_\alpha^N(\psi) - F_\alpha(\psi) \vert ^p\right]^{1/p} 
\leq \frac{C \norm{\psi}_\infty}{N^{1/2}} \, .
$$
In addition, the estimator is unbiased
$\bbE[F_{\alpha}^N(\psi)] = F_\alpha(\psi)$.
\end{prop}

Now, we define the function $\psi$ with respect to an arbitrary test function $\zeta_\alpha$, as follows
\begin{align}\label{eq:psidef}
\psi_{\zeta_\alpha}(\bx) & := \sum_{k=1}^{2^D} \iota_k \omega_k(\bx) \zeta_{\ba_k(\alpha)}(\bx_k) \, , \\ 
\omega_k(\bx)& := \frac{L_{\ba_k(\alpha)}\left (\bx_{k} \right)}
{{\bf L}_{\alpha}(\bx)} \, ,
\end{align}
where $\iota_k \in \{-1,1\}$ is the sign of the 
$k^{\rm th}$ term in 
$\Delta f_\alpha$ \footnote{Recall 
from equations \eqref{eq:increment}, \eqref{eq:incrementi} that $\Delta = \otimes_{i=1}^{D} \Delta_i = \Delta_D \circ \cdots \circ \Delta_1$ and $\Delta_i f_{\alpha} = f_{\alpha} - f_{\alpha-e_i}$.}. The function $\psi_{\zeta_{\alpha}}$ gives the mixed difference of the quantity of interest $\zeta_{\alpha}$ among $2^D$ intermediate multi-indices.
Of particular interest in our estimator are the functions
$\zeta_{\alpha} = \varphi_{\alpha}$, 
for arbitrary $\varphi_\alpha$,  
and $\zeta_{\alpha} = 1$.

\begin{ex}[Mixed Difference]
    Following from Example \ref{ex:approx_coupling}, an example of the mixed difference of the quantity of interest $\zeta_{(1,1)}$ constructed in \eqref{eq:psidef} is given by
    \begin{align*}
        \psi_{\zeta_{(1,1)}}(\bx) &= \bigg(\frac{L_{11}(\bx_4)}{{\bf L}_{(1,1)}(\bx)} \zeta_{11}(\bx_4) - \frac{L_{10}(\bx_4)}{{\bf L}_{(1,1)}(\bx)} \zeta_{10}(\bx_3) \bigg) - \bigg(\frac{L_{01}(\bx_4)}{{\bf L}_{(1,1)}(\bx)} \zeta_{01}(\bx_2) - \frac{L_{00}(\bx_4)}{{\bf L}_{(1,1)}(\bx)} \zeta_{00}(\bx_1) \bigg).
    \end{align*}
    Note that the signs of the terms in the mixed difference are $\iota_1 = 1$, $\iota_2 = -1$, $\iota_3 = -1$ and $\iota_4 = 1$.
\end{ex}

Following from Proposition \ref{prop:unnocoupled} we have that
\begin{equation}\label{eq:coupledbias}
\bbE[F_\alpha^N(\psi_{\zeta_\alpha})] = F_\alpha(\psi_{\zeta_\alpha}) 
= \Delta f_\alpha(\zeta_\alpha) \, ,
\end{equation}
and there is a $C>0$ such that
\begin{equation}\label{eq:coupledrate}
\bbE\left[\vert F_\alpha^N(\psi_{\zeta_\alpha}) - F_\alpha(\psi_{\zeta_\alpha}) \vert ^2\right] \leq C \frac{\norm{\psi_{\zeta_\alpha}}_\infty}{N} \, .
\end{equation}

Now given $\cI \subset \bbZ_+^D$ and $\{N_\alpha\}_{\alpha \in \cI}$, and $\varphi :\sX \rightarrow \bbR$, for each $\alpha$ run an independent SMC sampler as in Algorithm \ref{algo:smc} with $N_\alpha$ samples, and define the MIMC estimator as 
\begin{equation}\label{eq:mismc}
\widehat\varphi^{\rm MI}_\cI = \frac{\sum_{\alpha \in \cI} F^{N_\alpha}_\alpha(\psi_{\varphi_\alpha})}
{\max\{\sum_{\alpha \in \cI} F^{N_\alpha}_\alpha(\psi_{1}), Z_{\rm min}\}} \, ,
\end{equation}
where $Z_{\rm min}$ is a lower bound on $Z$.

A finer analysis than provided in Proposition \ref{prop:unnocoupled} in order to achieve rigorous MIMC complexity results is shown in Theorem~\ref{thm:mainshangda} given in \cite{law2022multi}.

\begin{theorem}\label{thm:mainshangda}
Assume \ref{ass:G1}. Then for any $J\in\mathbb{N}$ there exists a $C>0$ such that for any $N\in\mathbb{N}$, $\psi : \sX^{2^D} \rightarrow \bbR$ bounded and
measurable and $\alpha\in\mathbb{Z}_+^D$
$$
\bbE_{\alpha} \left[ \vert  F_\alpha^N(\psi_{\zeta_\alpha}) - 
F_\alpha(\psi_{\zeta_\alpha}) \vert ^2\right] 
\leq \frac{C}{N} 
\int_\sX (\Delta( L_\alpha(x)\zeta_\alpha(x) ))^2 \pi_0(dx) \, , 
$$
where $\psi_{\zeta_\alpha}(\bx)$ is as \eqref{eq:psidef}.
\end{theorem}

\begin{proof} 
The result is proven in \cite{law2022multi}.
\end{proof}

\subsection{Random sample size version}

Consider drawing $N/N_{\rm min}$ i.i.d. samples
$\alpha_i \sim {\bf p}$, where ${\bf p}$ is a probability distribution on $\bbZ_+^D$ with 
${\bf p}(\alpha) =: p_{\alpha} > 0$, to be specified later.
Define the allocations ${\bf A} \in \bbZ_+^{D \times N/N_{\rm min}}$ 
by ${\bf A}_i = \alpha_i$, and the (scaled) counts for each $\alpha \in \bbZ_+^D$
by $N_\alpha = N_{\rm min} \#\{ i ; \alpha_i = \alpha\} \in \bbZ_+$, collectively 
denoted ${\bf N}$. 
Note that $\bbE N_\alpha = N p_\alpha$ and $N_\alpha \rightarrow N p_\alpha$.

Now consider constructing a MISMC estimator of the type  
in \eqref{eq:mismc} using a random number $N_\alpha$ of samples, 
$F_\alpha^{N_\alpha}(\psi_{\zeta_\alpha})$, and recall the 
properties \eqref{eq:coupledbias} and Theorem \ref{thm:mainshangda}.
Conditioned on ${\bf A}$, or equivalently conditioned on ${\bf N}$,
these properties still hold.
For $\zeta: \mathsf{X} \rightarrow \bbR$ define the estimator
\begin{equation}\label{eq:rmismc}
\widehat{F}^{\rm rMI}(\zeta) := 
\sum_{\alpha \in \bbZ_+^D} \frac{N_\alpha}{N p_\alpha}  
F_\alpha^{N_\alpha}(\psi_{\zeta_\alpha}) \, .
\end{equation}

\subsection{Theoretical results}
\label{sec:theory}

The following standard MISMC assumptions will be made.

\begin{ass}\label{ass:rate}
For any $\zeta: \sX \rightarrow \bbR$ bounded and Lipschitz, there exist $C, \beta_i, s_i, \gamma_i >0$ for $i=1,\dots, D$ such that for resolution vector $(2^{-\alpha_1},\dots, 2^{-\alpha_D})$, i.e. resolution $2^{-\alpha_i}$ in the $i^{\rm th}$ direction, the following holds
\begin{itemize}
\item[(B)] $\vert \Delta f_\alpha(\zeta) \vert =: B_{\alpha} \leq C \prod_{i=1}^D 2^{-\alpha_i s_i}$;
\item[(V)] $\int_\sX (\Delta( L_\alpha(x)\zeta_\alpha(x) ))^2 \pi_0(dx) =: V_{\alpha} \leq C \prod_{i=1}^D 2^{-\alpha_i \beta_i}$;
\item[(C)] ${\rm COST}(F_\alpha(\psi_{\zeta})) 
\propto \prod_{i=1}^D 2^{\alpha_i \gamma_i}$.
\end{itemize}
\end{ass}

First, we need to examine the bias of the estimator \eqref{eq:rmismc}.

\begin{prop}\label{prop:unbiased}
Assume \ref{ass:G1}, and let $\zeta: \sX \rightarrow \bbR$. For any multi-index $\alpha$, we have $p_{\alpha} >0$. 
Then the randomized MISMC estimator \eqref{eq:rmismc}
is free from discretization bias, i.e.
\begin{equation}\label{eq:runbiased}
\bbE [ \widehat{F}^{\rm rMI}(\zeta) ] = f(\zeta) \, .
\end{equation}
\end{prop}
\begin{proof}
The proof is given in Appendix \ref{app:prop_unbiased}.
\end{proof}

Now that unbiasedness has been established, the next step is to examine the variance.
\begin{prop}\label{prop:variance}
Assume \ref{ass:G1} and \ref{ass:rate}, and let $\zeta: \sX \rightarrow \bbR$. For any multi-index $\alpha$, we have $p_{\alpha} >0$. Then the variance of the randomized MISMC estimator \eqref{eq:rmismc}
is given by
\begin{equation}
\bbE [ (\widehat{F}^{\rm rMI}(\zeta) - f(\zeta))^2 ] \leq  \frac{C}{N} \Bigg(
\sum_{\alpha \in \bbZ_{+}^{D}} \frac1{p_\alpha} \prod_{i=1}^D 2^{-\alpha_i \beta_i} 
+ \sum_{\alpha' \neq \alpha \in \bbZ_{+}^{D}} \prod_{i=1}^D 2^{-\alpha_i \beta_i/2}  
\prod_{j=1}^D 2^{-\alpha_j' \beta_j/2} \Bigg).\end{equation}
In particular, if 
$$
\sum_{\alpha \in \bbZ_{+}^{D}} \frac1{p_\alpha} \prod_{i=1}^D 2^{-\alpha_i \beta_i} \leq C \, ,
$$
then one has the canonical convergence rate.
\end{prop}

\begin{proof}
The proof is given in Appendix \ref{app:prop_variance}.
\end{proof}

The randomized MISMC ratio estimator is now defined
for $\varphi: \sX \rightarrow \bbR$ by
\begin{equation}\label{eq:rmismcrat}
\widehat{\varphi}^{\rm rMI} := 
\frac{\widehat{F}^{\rm rMI}(\varphi)}{\max\{\widehat{F}^{\rm rMI}(1),Z_{\rm min}\}} \, ,
\end{equation}
where $\widehat{F}^{\rm rMI}$ is defined in \eqref{eq:rmismc}.

Before presenting the main result of the present work, 
we first recall the main result of \cite{law2022multi} which is derived by Theorem~\ref{thm:mainshangda}. \cite{law2022multi} considers two index sets for the original MISMC ratio estimator, tensor product index set and total degree index set. Compared to the tensor product index set, the total degree index set abandons some expensive indices, with much looser conditions in the convergence theorem. The convergence result for tensor product index set is given in Theorem~\ref{thm:mainold_tp} of \cite{law2022multi} 
and for total degree index set it is given in Theorem~\ref{thm:mainold_td} of \cite{law2022multi}.

\begin{theorem}\label{thm:mainold_tp}
Assume \ref{ass:G1} and \ref{ass:rate}, with $\sum_{j=1}^D \frac{\gamma_j}{s_j} \leq 2$ 
and $\beta_i>\gamma_i$ for $i=1,\dots,D$. 
Then for any $\varepsilon>0$ and bounded and Lipschitz $\varphi: \sX \rightarrow \bbR$, it is possible to choose a tensor product index set $\cI_{L_1:L_D} := \{\alpha \in \bbZ_{+}^{D}:\alpha_1\in\{0,...,L_1\},...,\alpha_D \in \{0,...,L_D\}\}$ and $\{N_\alpha\}_{\alpha \in \cI_{L_1:L_D}}$,
such that for some $C>0$ that is independent of $\varepsilon$
$$
\bbE [ ( \widehat \varphi^{\rm MI}_{\cI_{L_1:L_D}}  - \pi(\varphi) )^2 ] \leq C \varepsilon^2 \, ,
$$
and ${\rm COST}(\widehat \varphi^{\rm MI}_{\cI_{L_1:L_D}} ) \leq C \varepsilon^{-2}$, the canonical rate.
The estimator $\widehat \varphi^{\rm MI}_{\cI_{L_1:L_D}} $ is defined in equation \eqref{eq:mismc}.
\end{theorem}

\begin{proof}
The proof is given in \cite{law2022multi}.
\end{proof}

\begin{theorem}\label{thm:mainold_td}
Assume \ref{ass:G1} and \ref{ass:rate}, with $\beta_i>\gamma_i$ for $i=1,\dots,D$. 
Then for any $\varepsilon>0$ and bounded and Lipschitz $\varphi: \sX \rightarrow \bbR$, it is possible to choose a total degree index set $\cI_{L} := \{\alpha \in \bbZ_{+}^{D}:  \sum_{i=1}^{D} \delta_i\alpha_i \leq L, \sum_{i=1}^{D} \delta_i=1 \}$, $\delta_i \in (0,1]$ and $\{N_\alpha\}_{\alpha \in \cI_{L_1:L_D}}$,
such that for some $C>0$ that is independent of $\varepsilon$
$$
\bbE [ ( \widehat \varphi^{\rm MI}_{\cI_{L}}  - \pi(\varphi) )^2 ] \leq C \varepsilon^2 \, ,
$$
and ${\rm COST}(\widehat \varphi^{\rm MI}_{\cI_{L}} ) \leq C \varepsilon^{-2}$, the canonical rate.
The estimator $\widehat \varphi^{\rm MI}_{\cI_{L}} $ is defined in equation \eqref{eq:mismc}.
\end{theorem}
\begin{proof}
The proof is given in \cite{law2022multi}.
\end{proof}

\begin{theorem}\label{thm:main}
Assume \ref{ass:G1} 
and \ref{ass:rate} (V,C), with $\beta_i>\gamma_i$ for $i=1,\dots,D$. Then, for bounded and Lipschitz $\varphi: \sX \rightarrow \bbR$, it is possible to choose a probability distribution ${\bf p}$ on $\bbZ_+^D$ such that for some $C>0$ that is independent of $N$
$$ 
\bbE [ ( \widehat \varphi^{\rm rMI} - \pi(\varphi) )^2 ] \leq \frac{C}{N} \, ,
$$
and expected ${\rm COST}(\widehat \varphi^{\rm rMI}) \leq C N$, i.e. the canonical rate.
The estimator $\widehat \varphi^{\rm rMI}$ is defined in equation \eqref{eq:rmismcrat}.
\end{theorem}

\begin{proof}
The proof is given in Appendix \ref{app:thm_main}.
\end{proof}

The noticeable differences in Theorem \ref{thm:main} with respect to Theorem \ref{thm:mainold_tp} and \ref{thm:mainold_td} are that (i) discretization bias does not appear and so the bias rates as in Assumption \ref{ass:rate} (B) are not required, nor is the constraint related to them shown in Table \ref{table:canonical_cond}, and (ii) no index set $\cI$ needs to be selected since the estimator sums over $\alpha \in \bbZ_+^D$ (noting that many of these indices do not get populated).

\begin{table}[h]
\centering
\begin{tabular}{ |c|c|c|}
    \hline
      \diagbox[width=9em,height=2em]{Methods}{Conditions} &  $\sum_{i=1}^{D} \frac{\gamma_{i}}{s_i} \leq 2$ & Bias tuning\\ 
    \hline
    rMISMC & no & no \\ 
    \hline
    MISMC with TD & no & yes \\
   \hline
   MISMC with TP & yes & yes \\
   \hline
\end{tabular}
\caption{Comparison of {\em non-desirable} 
constraints required to achieve canonical complexity 
with MISMC and rMISMC.}
\label{table:canonical_cond}
\end{table}

\section{Numerical results}
\label{sec:numerics}

The problems considered here are the same as in \cite{law2022multi}, and we intend to compare our rMISMC ratio estimator with the original MISMC ratio estimator.

The codes used for the numerical results in this paper can be found in \url{https://github.com/liangxinzhu/rMISMCRE.git}.

\subsection{Verification of assumption}

Discussions in connection with the required Assumption~\ref{ass:rate} for the 2D PDE and 2D LGP models are revisited here. Verification of the 1D PDE model is naturally satisfied according to the discussion of the 2D PDE model. Propositions \ref{prop:likelihood}, 
\ref{prop:likelihood_pde} and \ref{prop:likelihood_lgp} 
and their proofs are given in \cite{law2022multi}. 

We define the mixed Sobolev-like norms as
\begin{equation}
    \norm{x}_{q}:=\norm{A^{q/2}x},
\end{equation}
where  $A=\sum_{k\in \bbZ^2} a_k\phi_k \otimes \phi_k$, 
for the orthonormal basis $\{\phi_k\}$ defined above \eqref{eqn:spectral_decay}, $a_k = k_1^2k_2^2$, and $\norm{\cdot}$ is the $L^2(\Omega)$ norm. Note that the approximation of the posteriors of the motivating problems have the form $\exp(\Phi_{\alpha}(x))$ for some $\Phi_{\alpha}:\sX\rightarrow \bbR$.

\begin{prop}\label{prop:likelihood}
Let $\sX$ be a Banach space with $D=2$ s.t. $\pi_0(\sX)=1$, with norm $\norm{\cdot}_{\sX}$. For all $\epsilon > 0$, there exists a $C(\epsilon) > 0$ such that the following holds for $\Phi_{\alpha} = \log(L_{\alpha})$ given by \eqref{eqn:finite_likelihood_lgp} or \eqref{eqn:finite_likelihood_pde}, respectively:
\begin{equation}
    \Delta \exp(\Phi(x_{\alpha})) \leq C(\epsilon)\exp(\epsilon\norm{x}_{\sX}^2)\big(\vert \Delta \Phi_{\alpha}(x_{\alpha})\vert 
+ \vert \Delta_1 \Phi_{\alpha-e_2}(x_{\alpha-e_2})\vert \vert \Delta_2 \Phi_{\alpha-e_1}(x_{\alpha-e_1})\vert \big).\end{equation}
\end{prop}

The variance rate required in Assumption~\ref{ass:rate} (V) for PDE and LGP models are verified following Proposition~\ref{prop:likelihood_pde} and~\ref{prop:likelihood_lgp}. However, it is difficult to give theoretical verification for the variance rate in the LGC model. Since it involves a factor of double exponentials, like $\exp(\int \exp(x(z))dz)$, the Fernique Theorem does not guarantee that such a term is finite. Instead we verify it numerically, which is given in the Appendix \ref{app:2}.

\begin{prop}\label{prop:likelihood_pde}
Let $u_{\alpha}$ be the solution of~(\ref{eqn:pde}) at the resolution $\alpha$, as described in Section~\ref{sec:pde_finite}, for $a(x)$ given by~(\ref{eqn:coef_pde}) and uniformly over $x \in [-1,1]^d$, and ${\sf f}(x) \in L^2(\Omega)$. Then there exists a $C>0$ such that
\begin{equation}
    \norm{\Delta u_{\alpha}(x)} \leq C 2^{-2(\alpha_1 + \alpha_2)}.
\end{equation}
\end{prop}
Since $L_{\alpha}(x) \leq C < \infty$ in the PDE problem by Assumption \ref{ass:G1}, the constant in Proposition~\ref{prop:likelihood} can be made uniform over $x$, so the variance rate in Assumption~\ref{ass:rate} is obtained.

\begin{prop}\label{prop:likelihood_lgp}
Let $x \sim \pi_0$, where $\pi_0$ is a Gaussian process of the form~(\ref{eqn:gaussian_pro}) with spectral decay corresponding to~(\ref{eqn:spectral_decay}), and let $x_{\alpha}$ correspond to truncation on the index set $\cA_{\alpha} = \cap_{i=1}^2\{ \vert k_i \vert \leq 2^{\alpha_i} \}$ as in~(\ref{eqn:gaussian_pro_app}). Then there is a $C>0$ such that for all $q<(\beta-1)/2$
\begin{equation}
    \norm{\Delta x_{\alpha}}^2 \leq C \norm{x}_{q}^2 2^{-2q(\alpha_1 + \alpha_2) }.
\end{equation}
Furthermore, this rate is inherited by the likelihood with $\beta_i=\beta$.
\end{prop}

\subsection{1D toy problem}
We consider a 1D PDE toy problem which has already been applied in \cite{jasra2021unbiased} and \cite{law2022multi}. Let the domain be $\Omega=[0,1]$, $a = 1$ and ${\sf f}=x$ in PDE~(\ref{eqn:pde}). This toy PDE problem has an analytical solution, $u(x)(z) = \frac{x}{2}(z^2-z)$. Given the quantity of interest $\varphi(x) = x^2$, we aim to compute the expectation of the quantity of interest $\bbE[\varphi(x)]$. In the following implementation, we take the observations at ten points in the interval $[0,1]$, which are $[0.1, 0.2,...,0.9,1]$, so the observations are generated by
\begin{equation}
    y_i = -0.5x^{*}(z_i^2-z_i) + \nu_i, \text{ for } i=1,...,10,
\end{equation}  
where $[z_1,...,z_{10}] = [0.1, 0.2,...,0.9,1]$, $x^*$ is sampled from $U[-1,1]$ and $\nu_i \sim \cN(0,0.2^2)$. The reference solution is computed as in \cite{jasra2021unbiased} and \cite{law2022multi}. From Figure~\ref{fig:1dtoy_emp}, we have $\alpha = 2$, $\beta = 4$. The value of $\gamma$ is $1$ because we use linear nodal basis functions in FEM and tridiagonal solver.

From Figure~\ref{fig:1dtoy_com}, the convergence behaviour for rMLSMC and MLSMC is nearly the same and the convergence rate for them is approximately $-1$ which is the canonical rate. The difference in performance between (r)MLSMC and SMC is the rate of convergence, where the convergence rate for SMC is approximately $-4/5$. With the same total computational cost, the MSE of (r)MLSMC is larger than SMC until the cost reaches $10^4$. We conclude that (r)MLSMC performs better than SMC in terms of the rate of convergence as expected.

\begin{figure}[H]
  \centering
  \includegraphics{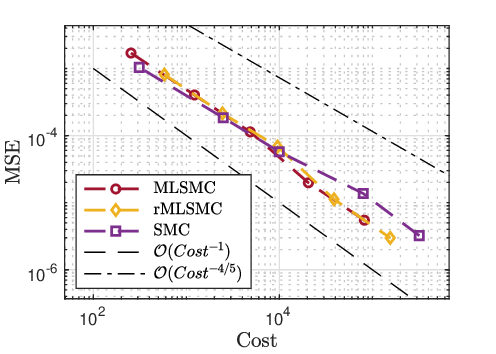} 
\caption{Comparison among three methods for the 1D inverse toy problem. We use 100 realisations to compute the MSE for each experiment and use the rate of convergence to compare the different methods. The red circle line is MLSMC with ratio estimator; the yellow diamond line is rMLSMC with ratio estimator; the purple square line is single-level sequential Monte Carlo. Rate of regression: (1)MLSMC: -1.008; (2)rMLSMC: -1.016; (3)SMC: -0.812}
\label{fig:1dtoy_com}
\end{figure}

\subsection{2D elliptic partial differential equation}

Applying rMLSMC in a 1D analytical PDE problem is only an appetizer, we now focus on applying rMISMC to high-dimensional problems. We now consider a 2D non-analytical elliptic PDE on the domain $\Omega=[0,1]^2$ with ${\sf f} = 100$ and $a(x)$ taking the form as
\begin{equation}
    a(x)(z) = 3+x_1 \cos(3z_1)\sin(3z_2) + x_2 \cos(z_1)\sin(z_2).
\end{equation}

We let the prior distribution be a uniform distribution $U[-1,1]^2$ and set the quantity of interest to be $\varphi(x) =  x_1^2 + x_2^2$, which is a generalisation of the one-dimensional case. We take the observations at a set of four points: \{(0.25, 0.25), (0.25,  0.75), (0.75, 0.25), (0.75, 0.75)\}, and the corresponding observations are given by 
\begin{equation}
    y = u_{\alpha}(x^{*}) + \nu,
\end{equation}
where $u_{\alpha}(x^{*})$ is the approximate solution at $\alpha = [10,10]$, $x^*$ samples from $U[-1,1]^2$ and $\nu \sim \cN(0,0.5^2{\bf I}_{4})$. The 2D PDE solver applied in this report is modified based on a MATLAB toolbox IFISS \cite{elman2007algorithm}.

Due to the zero Dirichlet boundary condition, the solution is zero at $\alpha_i=0$ and $\alpha_i=1$ for $i=1,2$. So we set the starting index as $\alpha_1=\alpha_2 = 2$. From Figure~\ref{fig:2dpde_emp}, we have $s_1=s_2=2$ and $\beta_1=\beta_2=4$ for the mixed rates. Since we use the bilinear basis function and MATLAB backslash code, one has $\gamma_1=\gamma_2=1$.

We consider two index sets in the MISMC approach, which are the tensor product (TP) index set and the total degree (TD) index set. From Figure~\ref{fig:2dpde_com}, MISMC with two different sets and rMISMC have similar convergence behaviour with convergence rate approximately being $-1$. Although we do not show SMC method in Figure~\ref{fig:2dpde_com}, the theoretical convergence rate of SMC will drop from $-4/5$ (1D) to $-2/3$ (2D), whose rate of convergence suffers the curse of dimensionality. Up to now, the convergence behaviour of MISMC (with TP index set or TD index set) and rMISMC is similar when applied to 1D and 2D PDE problems, which both achieve the canonical rates, but we will see a difference in the following two statistical models.

\begin{figure}[H]
  \centering
  \includegraphics{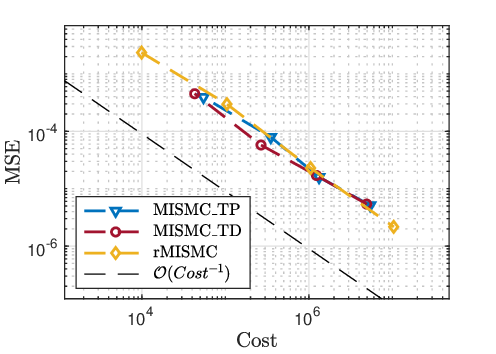} 
\caption{Comparison between two methods for the 2D non-analytical PDE. We use 200 realisations to compute the MSE for each experiment and use the rate of convergence to compare the different methods. The blue triangle line is MISMC with ratio estimator and tensor product set; the red circle line is MISMC with ratio estimator and total degree set; the yellow diamond line is rMISMC with ratio estimator. Rate of regression: (1)MISMC\_TP: -0.964; (2)MISMC\_TD: -0.925; (3)rMISMC: -1.015}
\label{fig:2dpde_com}
\end{figure}

\subsection{Log Gaussian Cox model}
Now, we consider the LGC model introduced in Section~\ref{sec:lgp}.
We set the parameters as $\theta = (\theta_1,\theta_2,\theta_3) = (0,1,(33/\pi)^2)$ in the LGC model. When using rMISMC and MISMC on the 2D log Gaussian Cox model, we need to set the starting level to $\alpha_1 = \alpha_2 = 5$, since the regularity shows up when $\alpha_1 \geq 5$ and $ \alpha_2 \geq 5$. Further, from Figure~\ref{fig:2dlgcp_emp}, we have mixed rates $s_1=s_2=0.8$ and $\beta_1=\beta_2=1.6$. Since we use the FFT for approximation, one has an asymptotic rate for the cost, $\gamma_i=1 + \omega < \beta_i$ for $\omega>0$ and $i=1,2$.

\begin{figure}[H]
  \centering
  \includegraphics{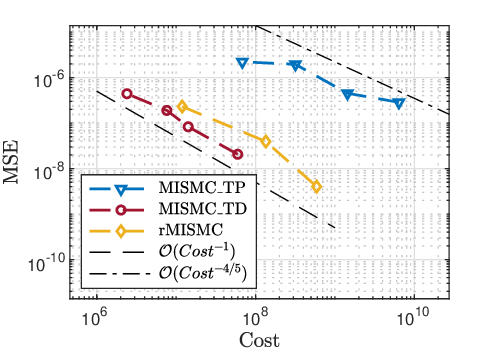} 
  \caption{Comparison between two methods for the 2D LGC model. We use 200 realisations to compute the MSE for each experiment and use the rate of convergence to compare the different methods. The blue triangle line is MISMC with ratio estimator and tensor product set; the red circle line is MISMC with ratio estimator and total degree set; the yellow diamond line is rMISMC with ratio estimator. Rate of regression: (1)MISMC\_TP: -0.502; (2)MISMC\_TD: -0.972; (3)rMISMC: -1.008}
\label{fig:2dlgcp_com}
\end{figure}

The rate of convergence of MISMC TD and rMISMC is approximately $-1$, and both of them achieve the canonical complexity of MSE$^{-1}$. However, the constant for MISMC TD is smaller than rMISMC. We have set a relatively large number of the minimum number of sample, $N_0$, in SMC sampler to alleviate the unexpected high variance caused by the few samples. It is reasonable to expect a higher variance
for the randomized method, however, since it involves infinitely many terms compared to finite. MISMC TD achieves a canonical rate, but MISMC TP only has a sub-canonical rate. This is because the assumption $ \sum_{i=1}^{2} \frac{\gamma_i}{s_i} \leq 2$ is violated ($\sum_{i=1}^{2} \frac{\gamma_i}{s_i} = \frac{5}{2}$) in MISMC, and this assumption is only needed in the tensor product index set, not in the total degree index set. This indicates that an improper choice of an index set in MISMC will result in dropping the canonical rate to the sub-canonical rate, which highlights the benefit of rMISMC since it achieves the canonical rate without providing an index set in advance.

\subsection{Log Gaussian process model}

We set the parameters as $\theta = (\theta_1, \theta_2, \theta_3) = (0,1,(33/\pi/2)^2)$ in the LGP model. Similar to the setting in LGP, when using rMISMC and MISMC on the 2D log Gaussian model, we need to set the starting level to $\alpha_1 = \alpha_2 = 5$, since the regularity shows when $\alpha_1 \geq 5$ and $ \alpha_2 \geq 5$. Further, from Figure~\ref{fig:2dlgp_emp}, we set $s_1=s_2=0.8$ and $\beta_1=\beta_2=1.6$. Same as the cost rate in LGC, one has $\gamma_i=1 + \omega < \beta_i$ for $\omega>0$ and $i=1,2$.

\begin{figure}[H]
  \centering
  \includegraphics{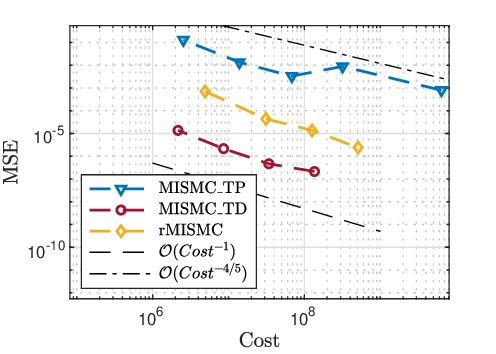}
\caption{Comparison between two methods for the 2D LGP Model. We use 200 realisations to compute the MSE for each experiment and use the rate of convergence to compare the different methods. The blue triangle line is MISMC with ratio estimator and tensor product set; the red circle line is MISMC with ratio estimator and total degree set; the yellow diamond line is rMISMC with ratio estimator. Rate of regression:  (1)MISMC\_TP: -0.565; (2)MISMC\_TD: -1.017; (3)rMISMC: -1.195}
\label{fig:2dlgp_com}
\end{figure}

In the LGP model, we can interpret similar results as in the LGC model: MISMC TP has the sub-canonical rate, and the constant for MISMC TD is smaller than rMISMC. However, the difference between constants for rMISMC and MISMC TD in LGP is much greater than in LGC. There may be other unidentified sources of high variance with respect to the rMISMC. 
This is the subject of ongoing investigation.
In addition, it should be noted that the LGP model is much more sensitive to parameter values ($\theta=(\theta_1,\theta_2,\theta_3)$) than the LGC model.

\subsubsection*{Financial Interests}

KJHL and XL gratefully acknowledge 
the support of IBM and EPSRC 
in the form of an Industrial Case Doctoral Studentship
Award.

\appendix

\section{Proofs}
\label{app:proofs}

The proofs of the various results in the paper are presented here, along with restatements of the results.

\subsection{Proof relating to Proposition \ref{prop:unbiased}}\label{app:prop_unbiased}
{\bf Proposition 2. }
{\em Assume \ref{ass:G1}, and let $\zeta: \sX \rightarrow \bbR$. For any multi-index $\alpha$, we have $p_{\alpha} >0$. Then the randomized MISMC estimator \eqref{eq:rmismc} is free from discretization bias, i.e.
\begin{equation}
\bbE [ \widehat{F}^{\rm rMI}(\zeta) ] = f(\zeta) \, .
\end{equation}}

\begin{proof}

Using the law of iterated conditional expectations,
and \eqref{eq:coupledbias} conditioned on $N_\alpha$,
one has
\begin{eqnarray}
&&\bbE[\widehat{F}^{\rm rMI}(\zeta)]\\\label{eq:runbiasedproof}
&=& \sum_{\alpha \in \bbZ_+^D} \sum_{N_\alpha\in \bbZ_+} 
\frac{\bbP(N_\alpha) N_\alpha}{N p_\alpha}
\bbE[F_\alpha^{N_\alpha}(\psi_{\zeta_\alpha}) \vert N_\alpha], \\
&=& \sum_{\alpha \in \bbZ_+^D} \frac{\bbE[N_\alpha]}{N p_\alpha}
\Delta f_\alpha(\zeta_\alpha), \\
&=& \sum_{\alpha \in \bbZ_+^D} 
\Delta f_\alpha(\zeta_\alpha) = f(\zeta) .
\end{eqnarray}
\end{proof}

\subsection{Proof relating to Proposition \ref{prop:variance}}\label{app:prop_variance}

{\bf Proposition 3. }
{\em Assume \ref{ass:G1} and \ref{ass:rate}, and let $\zeta: \sX \rightarrow \bbR$. For any multi-index $\alpha$, we have $p_{\alpha} >0$. Then the variance of the randomized MISMC estimator \eqref{eq:rmismc}
is given by
\begin{equation}
\bbE [ (\widehat{F}^{\rm rMI}(\zeta) - f(\zeta))^2 ] \leq  \frac{C}{N} \Bigg(
\sum_{\alpha \in \bbZ_{+}^{D}} \frac1{p_\alpha} \prod_{i=1}^D 2^{-\alpha_i \beta_i} 
\end{equation}
$$+ \sum_{\alpha' \neq \alpha \in \bbZ_{+}^{D}} \prod_{i=1}^D 2^{-\alpha_i \beta_i/2}  
\prod_{j=1}^D 2^{-\alpha_j' \beta_j/2} \Bigg).$$
In particular, if 
$$
\sum_{\alpha \in \bbZ_{+}^{D}} \frac1{p_\alpha} \prod_{i=1}^D 2^{-\alpha_i \beta_i} \leq C \, ,
$$
then one has the canonical convergence rate.}

\begin{proof}

One has
\begin{eqnarray}\nonumber
&& \bbE [ (\widehat{F}^{\rm rMI}(\zeta) - f(\zeta))^2 ] \nonumber\\
&& =  
\sum_{\alpha \in \bbZ_+^D} \bbE \left [ \left( \frac{N_\alpha}{N p_\alpha} 
F_\alpha^{N_\alpha}(\psi_{\zeta_\alpha}) - \Delta f_\alpha(\zeta_\alpha) \right)^2\right ] \nonumber\\
&& + \sum_{\alpha' \neq \alpha\in \bbZ_+^D} \bbE \Bigg[ \left(\frac{N_\alpha}{N p_\alpha}  
F_\alpha^{N_\alpha}(\psi_{\zeta_\alpha}) - \Delta f_\alpha(\zeta_\alpha) \right) \nonumber\\
&& \left (  \frac{N_{\alpha'}}{N p_{\alpha'}} 
F_{\alpha'}^{N_{\alpha'}}(\psi_{\zeta_{\alpha'}}) - \Delta f_{\alpha'}(\zeta_{\alpha'})  \right)\Bigg] . \nonumber
\end{eqnarray}
The diagonal and off-diagonal terms will be treated separately. First, for the diagonal terms, adding $\pm \frac{N_\alpha}{N p_\alpha} \Delta f_\alpha(\zeta_\alpha)$ in the square term of the diagonal term and using the triangle inequality, we have
\begin{equation}\label{eqn:triangle_ineq}
\begin{aligned}
 &   \sum_{\alpha \in \bbZ_+^D} \bbE \left [ \left( \frac{N_\alpha}{N p_\alpha} 
F_\alpha^{N_\alpha}(\psi_{\zeta_\alpha}) - \Delta f_\alpha(\zeta_\alpha) \right)^2\right ] \\
&\leq
\sum_{\alpha \in \bbZ_+^D} \Bigg( \sqrt{ \bbE \left [ \left( \frac{N_\alpha}{N p_\alpha} \right)^2  \left( F_\alpha^{N_\alpha}(\psi_{\zeta_\alpha}) - \Delta f_{\alpha}(\zeta_{\alpha}) \right)^2\right ]} \\
& + \sqrt{\bbE \left [ \left( \frac{N_\alpha}{N p_\alpha} - 1 \right)^2 \left( \Delta f_{\alpha}(\zeta_{\alpha}) \right)^2\right ]} \Bigg)^2
\end{aligned}
\end{equation}

Since $\bf p$ is proportional to a multinomial distribution, we have $\bbE[N_{\alpha}^2] = Np_{\alpha}(1-p_{\alpha})+N^2p_{\alpha}^2$. Then, for the first term of (\ref{eqn:triangle_ineq}), the same decomposition of \eqref{eq:runbiasedproof},
along with the result of Theorem \ref{thm:mainshangda} 
(conditioned on $N_\alpha$) and Assumption \ref{ass:rate} (V) gives
\begin{equation}\label{eqn:rewrite}
\bbE \left [ \left( \frac{N_\alpha}{N p_\alpha} \right)^2  \left( F_\alpha^{N_\alpha}(\psi_{\zeta_\alpha}) - \Delta f_{\alpha}(\zeta_{\alpha}) \right)^2\right ]
\leq 
\frac{C}{N p_\alpha} \prod_{i=1}^D 2^{-\alpha_i \beta_i}.
\end{equation}
For the second term of (\ref{eqn:triangle_ineq}), Applying Jensen's inequality to $\left(\Delta f_{\alpha}(\zeta_{\alpha})\right)^2$, we have
\begin{equation}
\bbE \left [ \left( \frac{N_\alpha}{N p_\alpha} - 1 \right)^2 \left( \Delta f_{\alpha}(\zeta_{\alpha}) \right)^2\right ] \leq \frac{C}{N p_\alpha} \prod_{i=1}^D 2^{-\alpha_i \beta_i}.
\end{equation}

Hence, we derive the bound for the diagonal term as follows
\begin{equation}
    \sum_{\alpha \in \bbZ_+^D} \bbE \left [ \left( \frac{N_\alpha}{N p_\alpha} 
F_\alpha^{N_\alpha}(\psi_{\zeta_\alpha}) - \Delta f_\alpha(\zeta_\alpha) \right)^2\right ] 
\end{equation}
$$\leq \frac{C}{N}\sum_{\alpha \in\bbZ_{+}^{D}} \frac{1}{p_{\alpha}} \prod_{i=1}^D 2^{-\alpha_i \beta_i}.$$

Now, the off-diagonal terms are more subtle because of the correlation between 
$N_\alpha$ and $N_{\alpha'}$, which are (proportional to) 
draws from a multinomial distribution with $N/N_{\rm min}$ samples and probability
${\bf p}$, and so
\begin{equation}\label{eq:multinomial}
\bbE [ (N_\alpha - N p_\alpha)(N_{\alpha'} - Np_{\alpha'}) ] = - N p_\alpha p_{\alpha'} \, .
\end{equation}
Using the law of iterated conditional expectations
$$
\bbE \Bigg[ \left(\frac{N_\alpha}{N p_\alpha}  
F_\alpha^{N_\alpha}(\psi_{\zeta_\alpha}) - \Delta f_\alpha(\zeta_\alpha) \right)
$$
$$
\left (  \frac{N_{\alpha'}}{N p_{\alpha'}}  
F_{\alpha'}^{N_{\alpha'}}(\psi_{\zeta_{\alpha'}}) - \Delta f_{\alpha'}(\zeta_{\alpha'}) \right)\Bigg] , 
$$
$$
=
\bbE \Bigg[ \bbE \Bigg [ \left(\frac{N_\alpha}{N p_\alpha}  
F_\alpha^{N_\alpha}(\psi_{\zeta_\alpha}) - \Delta f_\alpha(\zeta_\alpha) \right)
$$
$$
\left (  \frac{N_{\alpha'}}{N p_{\alpha'}}  
F_{\alpha'}^{N_{\alpha'}}(\psi_{\zeta_{\alpha'}}) - \Delta f_{\alpha'}(\zeta_{\alpha'}) \right)
\Bigg \vert {\bf N} \Bigg] \Bigg ],
$$
\begin{align}
&= \frac{\Delta f_\alpha(\zeta_\alpha)\Delta f_{\alpha'}(\zeta_{\alpha'})}
{N^2 p_\alpha p_{\alpha'}} 
\bbE [ (N_\alpha - N p_\alpha)(N_{\alpha'} - Np_{\alpha'}) ], \label{eq:inde}
\\ 
&\leq \frac{C}{N} \prod_{i=1}^D 2^{-\alpha_i \beta_i/2} 
\prod_{j=1}^D 2^{-\alpha'_j \beta_j/2} 
. \label{eq:multiuse}
\end{align}
Line \eqref{eq:inde} follows from the conditional independence
of $F_\alpha^{N_\alpha}(\psi_{\zeta_\alpha})$ and 
$F_{\alpha'}^{N_{\alpha'}}(\psi_{\zeta_{\alpha'}})$ 
given ${\bf N} = \{N_\alpha\}_{\alpha \in \bbZ_+^D}$
and some simple algebra.
Line \eqref{eq:multiuse} follows from \eqref{eq:multinomial}
and Assumption \ref{ass:rate} (V), 
along with Jensen's inequality applied to 
$(\Delta f_\alpha(\zeta_\alpha))^2$ and $(\Delta f_{\alpha'}(\zeta_{\alpha'}))^2$.

Furthermore, 
\begin{align}
    & \sum_{\alpha' \neq \alpha \in \bbZ_{+}^{D}} \prod_{i=1}^D 2^{-\alpha_i \beta_i/2}  
\prod_{j=1}^D 2^{-\alpha_j' \beta_j/2} \\
& = \Big(\sum_{\alpha \in \bbZ_{+}^{D}} \prod_{i=1}^D 2^{-\alpha_i \beta_i/2} \Big)^2 - \sum_{\alpha \in \bbZ_{+}^{D}} \prod_{i=1}^D 2^{-\alpha_i \beta_i}, \nonumber \\
& = \Big(\prod_{i=1}^D \frac{1}{1-2^{- \beta_{i}/2}}\Big)^2 - \prod_{i=1}^D \frac{1}{1-2^{- \beta_{i}}}, \nonumber \\
& = \prod_{i=1}^D \frac{1}{(1-2^{- \beta_{i}/2})^2} - \prod_{i=1}^D \frac{1}{1-2^{- \beta_{i}}} < C.
\end{align}
Since $\beta_i >0$, $1-2^{-\beta_i} <1$ and $1-2^{-\beta_i/2} <1$, we can found a constant $C$ to bound the last line.
\end{proof}

\subsection{Proof relating to Theorem \ref{thm:main}}\label{app:thm_main}

{\bf Theorem 4. }
{\em Assume \ref{ass:G1} 
and \ref{ass:rate} (V,C), with $\beta_i>\gamma_i$ for $i=1,\dots,D$. Then, for suitable $\varphi: \sX \rightarrow \bbR$, it is possible to choose a probability distribution ${\bf p}$ on $\bbZ_+^D$ such that for some $C>0$ 
$$
\bbE [ ( \widehat \varphi^{\rm rMI} - \pi(\varphi) )^2 ] \leq \frac{C}{N} \, ,
$$
and expected ${\rm COST}(\widehat \varphi^{\rm rMI}) \leq C N$, i.e. the canonical rate.
The estimator $\widehat \varphi^{\rm rMI}$ is defined in equation \eqref{eq:rmismcrat}.}

\begin{proof}
Since our unnormalized estimator is unbiased, we have
\begin{equation}
    \bbE\big[(\widehat{\varphi}^{\text{rMI}} - \pi(\varphi))^2\big] \leq \max_{\zeta \in \{\varphi,1\}} \bbE\big[ (\widehat{F}^{\rm rMI}(\zeta) - f(\zeta))^2 \big].
\end{equation}
Then, $\bbE[(\hat{\varphi}^{\text{rMI}} - \pi(\varphi))^2]$ is less than $C/N$ as long as $\max_{\zeta \in \{\varphi,1\}} \bbE\big[ (\widehat{F}^{\rm rMI}(\zeta) - f(\zeta))^2 \big]$ is of $\cO(N^{-1})$. 

Following from Proposition~\ref{prop:variance}, we need
\begin{equation}\label{eqn:thm_var}
    \sum_{\alpha \in \bbZ_{+}^{D}} \frac1{p_\alpha} \prod_{i=1}^D 2^{-\alpha_i \beta_i} \leq C
\end{equation}
to let $\bbE\big[ (\widehat{F}^{\rm rMI}(\zeta) - f(\zeta))^2 \big]$ be of $\cO(N^{-1})$. Additionally, we let the distribution follows from a simple constrained minimization of
the expected cost for one sample
\begin{equation}\label{eqn:thm_cost}
    \sum_{\alpha \in \bbZ_+^D} p_\alpha \prod_{i=1}^D 2^{\alpha_i \gamma_i} \leq C.
\end{equation}
Since $p_{\alpha}>0$, one finds that
\begin{equation}
    p_\alpha \propto \prod_{i=1}^D 2^{-\alpha_i (\gamma_i+\beta_i)/2}
\end{equation}
is one suitable choice for probabilities to satisfy the two conditions stated in (\ref{eqn:thm_var}) and (\ref{eqn:thm_cost}). Then the expected cost ${\rm COST}(\widehat \varphi^{\rm rMI}) = 
\sum_{\alpha \in \bbZ_+^D} N p_\alpha \prod_{i=1}^D 2^{\alpha_i \gamma_i}.$ and (\ref{eqn:thm_var}) are then given respectively by
\begin{align}
& \sum_{\alpha \in \bbZ_+^D} N \prod_{i=1}^D 2^{\alpha_i (\gamma_i-\beta_i)/2}, \\
& \sum_{\alpha \in \bbZ_+^D} \prod_{i=1}^D 2^{\alpha_i (\gamma_i-\beta_i)/2}.
\end{align}
By assumption $\gamma_i-\beta_i <0$, the expected cost is bounded by $CN$ and (\ref{eqn:thm_var}) is finite. Hence, $\bbE\big[ (\widehat{F}^{\rm rMI}(\zeta) - f(\zeta))^2 \big]$ is of $\cO(N^{-1})$.
\end{proof}

\section{Figures}
\label{app:2}

The plots of best log linear fit of convergence parameters for the various problems are illustrated in Figures \ref{fig:1dtoy_emp}, \ref{fig:2dpde_emp}, \ref{fig:2dlgcp_emp} and \ref{fig:2dlgp_emp}.

\begin{figure*}[t]
  \centering
  \includegraphics{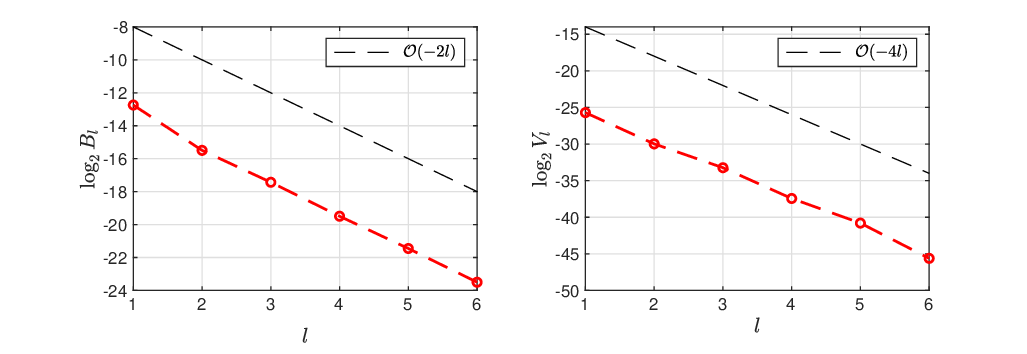}  
\caption{Convergence parameters fitting for the 1D inverse toy problem. Empirical values for strong and weak convergence parameters. All plots use 100 realizations with 1000 samples in each index as the maximum level $L_{\max}=6$ at each experiment. (Left) Logarithm plot with base 2 for the mean of difference and reference line with slope -2; (Right) Logarithm plot with base 2 for the variance of difference and reference line with slope -4}
\label{fig:1dtoy_emp}
\end{figure*}

\begin{figure*}[t]
  \centering
   \includegraphics{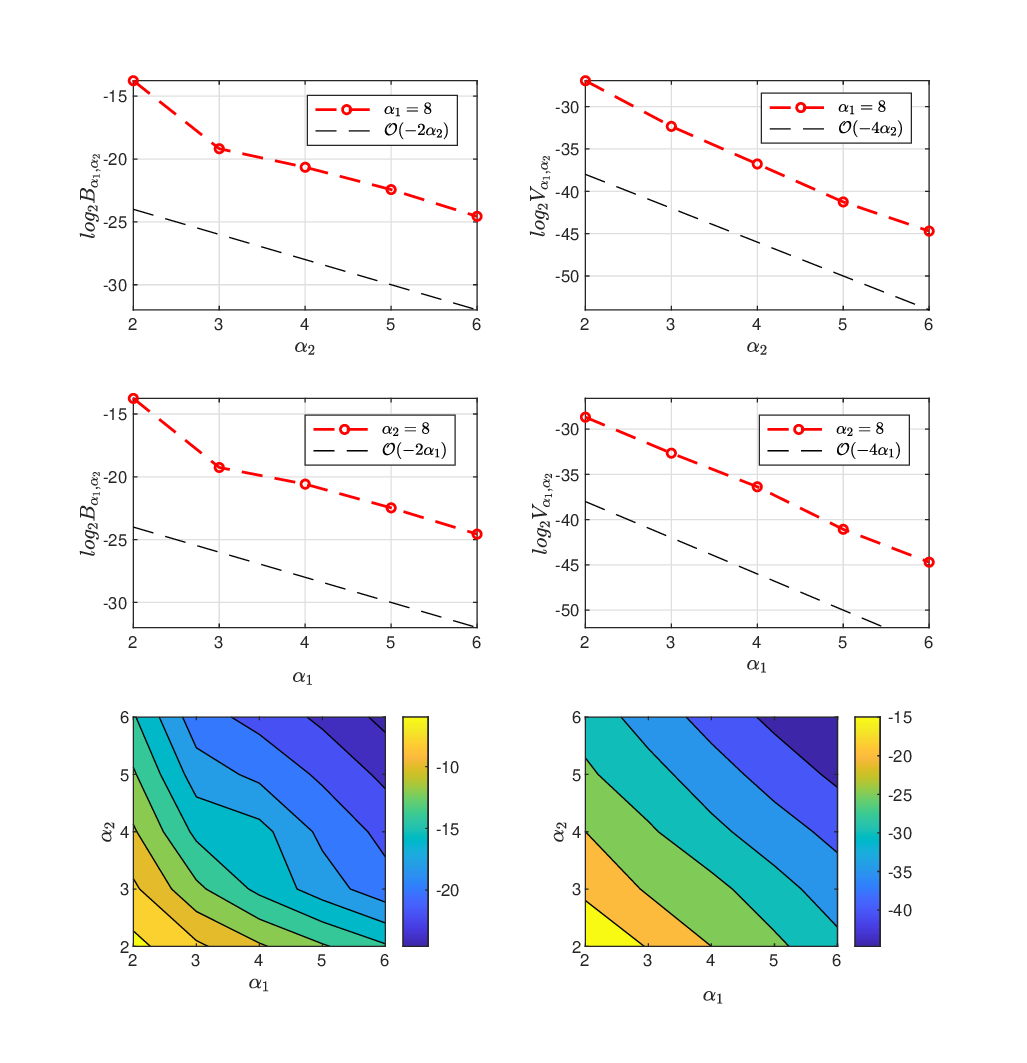}  
\caption{Convergence parameters fitting for the 2D inverse elliptic PDE problem. Empirical values for strong and weak convergence parameters. All plots use 20 realizations with 1000 samples in each index as the maximum multi-index $\alpha_1=6$ and $\alpha_2=6$ for each experiment. (Left) Logarithm plots with base 2 for the mean of unnormalized increments of increments; (Right) logarithm plots with base 2 for the variance of unnormalized increments of increments}
\label{fig:2dpde_emp}
\end{figure*}

\begin{figure*}[t]
  \centering
   \includegraphics{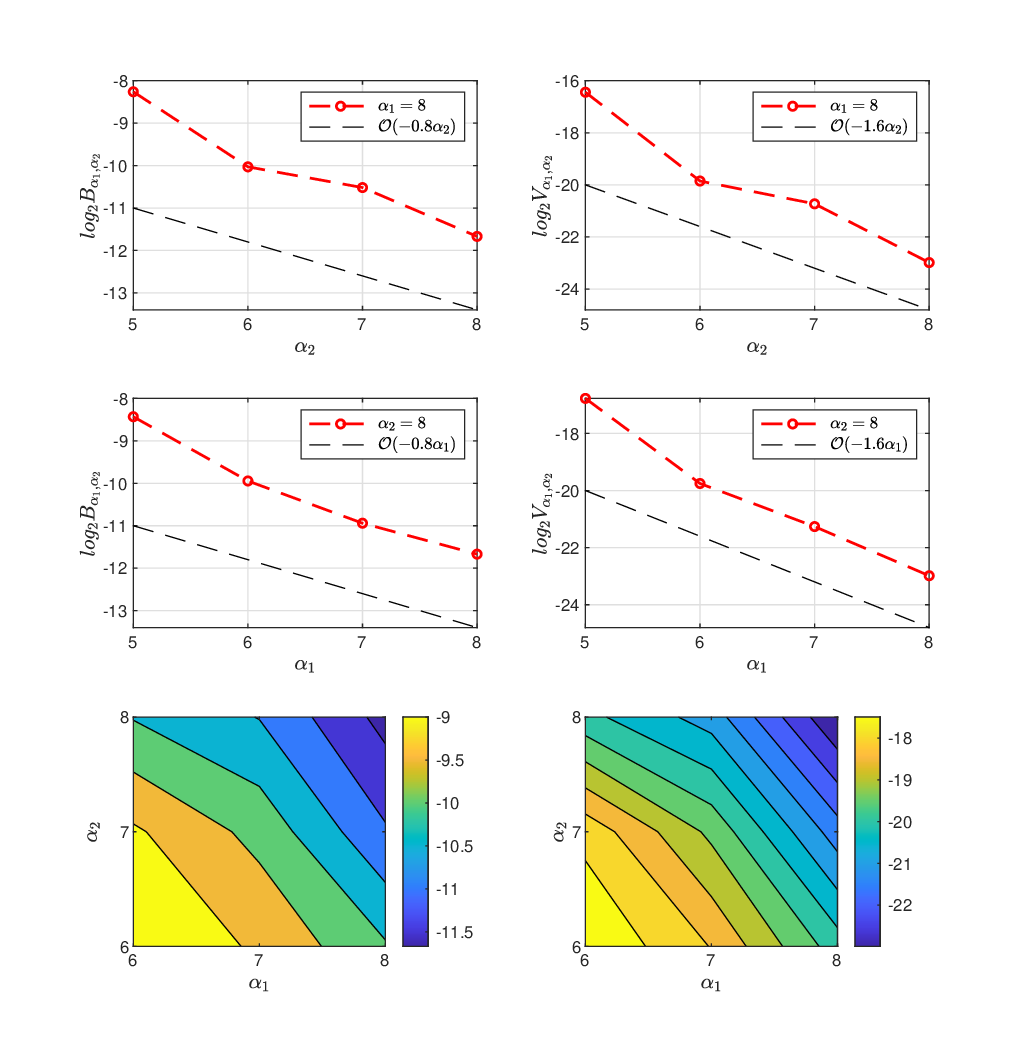}
\caption{Convergence parameters fitting for the 2D LGC model. Empirical values for strong and weak convergence parameters. All plots use 20 realizations with 1000 samples in each index as the minimum multi-index $\alpha_1=\alpha_2=5$ and maximum multi-index $\alpha_1=\alpha_2=8$ at each experiment. (Left) Logarithm plots with base 2 for the mean of unnormalized increments of increments; (Right) logarithm plots with base 2 for the variance of unnormalized increments of increments}
\label{fig:2dlgcp_emp}
\end{figure*}

\begin{figure*}[t]
  \centering
   \includegraphics{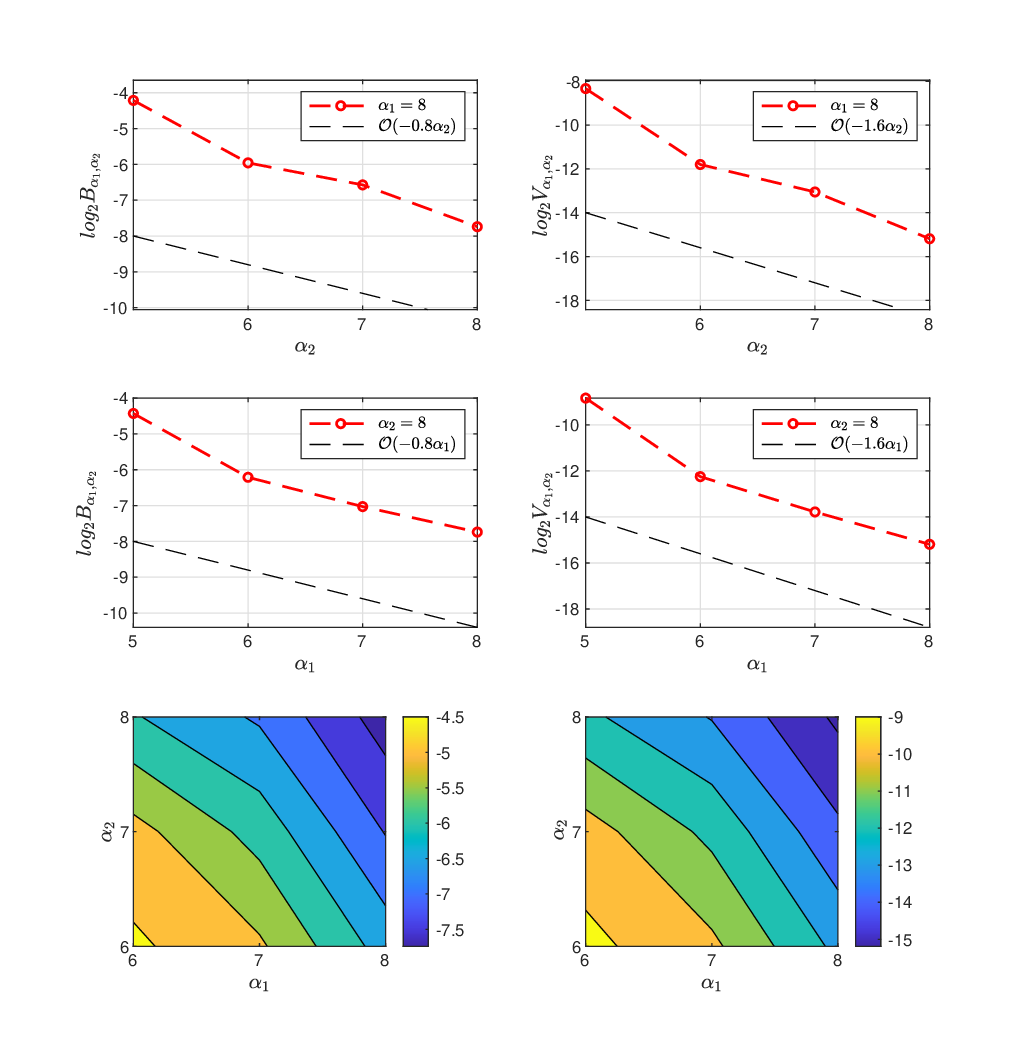}  
\caption{Convergence parameters fitting for the 2D LGP model. Empirical values for strong and weak convergence parameters. All plots use 20 realizations with 1000 samples in each index as the minimum multi-index $\alpha_1=\alpha_2=5$ and maximum multi-index $\alpha_1=\alpha_2=8$ at each experiment. (Left) Logarithm plots with base 2 for the mean of unnormalized increments of increments; (Right) logarithm plots with base 2 for the variance of unnormalized increments of increments}
\label{fig:2dlgp_emp}
\end{figure*}

%\bibliographystyle{abbrvnat}
%\bibliography{refs}
\bibliography{refs}
\bibliographystyle{plain}

\end{document}